\documentclass[11pt]{amsart}
\usepackage{latexsym,amsfonts,amssymb}
\usepackage{hyperref}

\newtheorem{theorem}{Theorem}[section]
\newtheorem{lemma}[theorem]{Lemma}
\newtheorem{corollary}[theorem]{Corollary}
\newtheorem{question}[theorem]{Question}
\theoremstyle{definition}
\newtheorem{definition}[theorem]{Definition}

\theoremstyle{remark}

\begin{document}
\noindent \vspace{0.5in}

\title[Sequence-covering maps on generalized metric spaces]%
{Sequence-covering maps on generalized metric spaces}

\author{Fucai\ \ Lin*}
\thanks{*The corresponding author.}
\address{(Fucai Lin)Department of Mathematics,
Zhangzhou Normal University, Zhangzhou 363000, P. R.China}
\email{linfucai2008@yahoo.com.cn}
\author{Shou Lin}
\address{(Shou Lin)Institute of Mathematics, Ningde Teachers' College, Ningde, Fujian
352100, P. R. China} \email{linshou@public.ndptt.fj.cn}


\begin{abstract} Let $f:X\rightarrow Y$ be a map. $f$ is a {\it sequence-covering
map}\cite{Si1} if whenever $\{y_{n}\}$ is a convergent sequence in
$Y$ there is a convergent sequence $\{x_{n}\}$ in $X$ with each
$x_{n}\in f^{-1}(y_{n})$; $f$ is an {\it 1-sequence-covering
map}\cite{Ls2} if for each $y\in Y$ there is $x\in f^{-1}(y)$ such
that whenever $\{y_{n}\}$ is a sequence converging to $y$ in $Y$
there is a sequence $\{x_{n}\}$ converging to $x$ in $X$ with each
$x_{n}\in f^{-1}(y_{n})$. In this paper, we mainly discuss the
sequence-covering maps on generalized metric spaces, and give an
affirmative answer for a question in \cite{LL1} and some related
questions, which improve some results in \cite{LL1, Ls4, YP},
respectively. Moreover, we also prove that open and closed maps
preserve strongly monotonically monolithity, and closed
sequence-covering maps preserve spaces with a
$\sigma$-point-discrete $k$-network. Some questions about
sequence-covering maps on generalized metric spaces are posed.
\end{abstract}

\maketitle

\section{Introduction}
A study of images of topological spaces under certain
sequence-covering maps is an important question in general topology
\cite{GMT, LJ, Li, LL1, Ls3, LY, LZGG, LC2, YP}. S. Lin and P. F.
Yan in \cite{LY} proved that each sequence-covering and compact map
on metric spaces is an 1-sequence-covering map. Recently, F. C. Lin
and S. Lin in \cite{LL1} proved that each sequence-covering and
boundary-compact map on metric spaces is an 1-sequence-covering map.
Also, the authors posed the following question in \cite{LL1} :

\begin{question}\cite[Question 3.6]{LL1}
Let $f:X\rightarrow Y$ be a sequence-covering and boundary-compact
map. Is $f$ an 1-sequence-covering map if $X$ is a space with a
point-countable base or a developable space?
\end{question}

In this paper, we shall give an affirmative answer for Question 1.1.

S. Lin in \cite[Theorem 2.2]{Ls4} proved that if $X$ is a metrizable
space and $f$ is a sequence-quotient and compact map, then $f$ is a
pseudo-sequence-covering map. Recently, C. F. Lin and S. Lin in
\cite{LL1} proved that if $X$ is a metrizable space and $f$ is a
sequence-quotient and boundary-compact map, then $f$ is a
pseudo-sequence-covering map. Hence we have the following Question
1.2.

\begin{question}
Let $f:X\rightarrow Y$ be a sequence-quotient and boundary-compact
map. Is $f$ a pseudo-sequence-covering map if $X$ is a space with a
point-countable base or a developable space?
\end{question}

On the other hand, the authors in \cite{YP} proved that each closed
sequence-covering map on metric spaces is an 1-sequence-covering
map. Hence we have the following Question 1.3.

\begin{question}
Let $f:X\rightarrow Y$ be a closed sequence-covering map. Is $f$ an
1-sequence-covering map if $X$ is a regular space with a
point-countable base or a developable space?
\end{question}

In this paper, we shall we give an affirmative answer for Question
1.2, which improves some results in \cite{LL1} and \cite{Ls4},
respectively. Moreover, we give an affirmative answer for Question
1.3 when $X$ has a point-countable base or $X$ is $g$-metirzable. In
\cite{TV}, V. V. Tkachuk introduced the strongly monotonically
monolithic spaces. In this paper, we also prove that strongly
monotonically monolithities are preserved by open and closed maps,
and spaces with a $\sigma$-point-discrete $k$-network are preserved
by closed sequence-covering maps.

\vskip 1cm\setlength{\parindent}{1cm}
\section{Definitions and terminology}
Let $X$ be a space. For $P\subset X$, $P$ is a {\it sequential
neighborhood} of $x$ in $X$ if every sequence converging to $x$ is
eventually in $P$.

\begin{definition}
Let $\mathcal{P}=\bigcup_{x\in X}\mathcal{P}_{x}$ be a cover of a
space $X$ such that for each $x\in X$, (a)\ if $U,V\in
\mathcal{P}_{x}$, then $W\subset U\cap V$ for some $W\in
\mathcal{P}_{x}$; (b)\ $\mathcal{P}_{x}$ is a network of $x$ in $X$,
i.e., $x\in\bigcap\mathcal{P}_x$, and if $x\in U$ with $U$ open in
$X$, then $P\subset U$ for some $P\in\mathcal P_x$.

(1)$\mathcal{P}$ is called an {\it $sn$-network} for $X$ if each
element of $\mathcal{P}_{x}$ is a sequential neighborhood of $x$ in
$X$ for each $x\in X$. $X$ is called {\it
$snf$-countable}\cite{Ls3}, if $X$ has an $sn$-network $\mathcal P$
such that each $\mathcal P_x$ is countable.

(2)$\mathcal{P}$ is called a {\it weak base}\cite{Ar} for $X$ if
whenever $G\subset X$ satisfying for each $x\in X$ there is a $P\in
\mathcal{P}_{x}$ with $P\subset G$,\ $G$ is open in $X$. $X$ is {\it
$g$-metrizable}\cite{Si2} if $X$ is regular and has a
$\sigma$-locally finite weak base.
\end{definition}

\begin{definition}
Let $f:X\rightarrow Y$ be a map.
\begin{enumerate}
\item $f$ is a {\it compact} (resp. {\it separable}) map if each
$f^{-1}(y)$ is compact (separable) in $X$;

\item $f$ is a {\it boundary-compact}(resp. {\it boundary-separable}) map if each $\partial f^{-1}(y)$
is compact (separable) in $X$;

\item $f$ is a {\it sequence-covering
map}\cite{Si1} if whenever $\{y_{n}\}$ is a convergent sequence in
$Y$ there is a convergent sequence $\{x_{n}\}$ in $X$ with each
$x_{n}\in f^{-1}(y_{n})$;

\item $f$ is an {\it 1-sequence-covering
map}\cite{Ls2} if for each $y\in Y$ there is $x\in f^{-1}(y)$ such
that whenever $\{y_{n}\}$ is a sequence converging to $y$ in $Y$
there is a sequence $\{x_{n}\}$ converging to $x$ in $X$ with each
$x_{n}\in f^{-1}(y_{n})$;

\item $f$ is a {\it sequentially quotient map}\cite{BS} if whenever
$\{y_{n}\}$ is a convergent sequence in $Y$ there is a convergent
sequence $\{x_{k}\}$ in $X$ with each $x_{k}\in f^{-1}(y_{n_{k}})$;

\item $f$ is a {\it pseudo-sequence-covering map}\cite{GMT, ILT} if for each
convergent sequence $L$ in $Y$ there is a compact subset $K$ in $X$
such that $f(K)=\overline{L}$;
\end{enumerate}
\end{definition}

It is obvious that

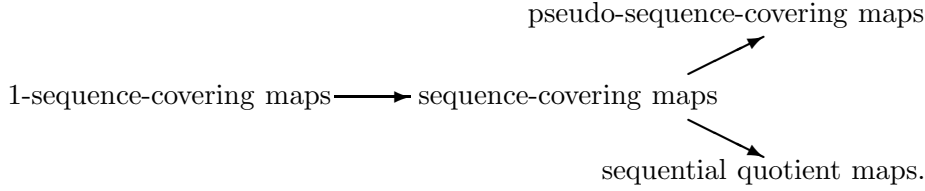
\setlength{\unitlength}{1cm}
\begin{picture}(15,1.5)\thicklines
 \put(2.1,0){\makebox(0,0){1-sequence-covering maps}}
 \put(4.3,0){\vector(1,0){1}}
 \put(7.4,0){\makebox(0,0){sequence-covering maps}}
 \put(9,0.3){\vector(2,1){1}}
 \put(9.5,1.1){\makebox(0,0){pseudo-sequence-covering maps}}
 \put(9,-0.3){\vector(2,-1){1}}
 \put(10,-1){\makebox(0,0){sequential quotient maps.}}
\end{picture}

\vskip 1.4cm Remind readers attention that the sequence-covering
maps defined the above-mentioned are different from the
sequence-covering maps defined in \cite{GMT}, which is called
pseudo-sequence-covering maps in this paper.

\begin{definition}\cite{MN}
Let $A$ be a subset of a space $X$. We call an open family
$\mathcal{N}$ of subsets of $X$ is an {\it external base} of $A$ in
$X$ if for any $x\in A$ and open subset $U$ with $x\in U$ there is a
$V\in \mathcal{N}$ such that $x\in V\subset U$.
\end{definition}

Similarly,  we can define an {\it externally weak base} for a subset
$A$ for a space $X$.

Throughout this paper all spaces are assumed to be Hausdorff, all
maps are continuous and onto. The letter $\mathbb{N}$ will denote
the set of positive integer numbers. Readers may refer to \cite{En,
Gr, Ls3} for unstated definitions and terminology.

\vskip 1cm\setlength{\parindent}{1cm}
\section{Sequence-covering and boundary-compact maps}
Let ${\it \Omega}$ be the sets of all topological spaces such that,
for each compact subset $K\subset X\in {\it \Omega}$, $K$ is
metrizable and also has a countably neighborhood base in $X$. In
fact, E. A. Michael and K. Nagami in \cite{MN} has proved that $X\in
{\it \Omega}$ if and only if $X$ is the image of some metric space
under an open and compact-covering\footnote{Let $f:X\rightarrow Y$
be a map. $f$ is called a {\it compact-covering map}\cite{MN} if in
case $L$ is compact in $Y$ there is a compact subset $K$ of $X$ such
that $f(K)= L$.} map. It is easy to see that if a space $X$ is
developable or has a point-countable base, then $X\in {\it \Omega}$
(see \cite{AB} and \cite{TV}, respectively).

In this paper, when we say an $snf$-countable space $Y$, it is
always assume that $Y$ has an $sn$-network
$\mathcal{P}=\cup\{\mathcal{P}_{y}:y\in Y\}$ such that
$\mathcal{P}_{y}$ is countable and closed under finite intersections
for each point $y\in Y$.

\begin{lemma}\label{l5}
Let $f:X\rightarrow Y$ be a sequence-covering and boundary-compact
map, where $Y$ is $snf$-countable. For each non-isolated point $y\in
Y$, there exists a point $x_{y}\in\partial f^{-1}(y)$ such that
whenever $U$ is an open subset with $x_{y}\in U$, there exists a
$P\in\mathcal{P}_{y}$ satisfying $P\subset f(U)$
\end{lemma}

\begin{proof}
Suppose not, there exists a non-isolated point $y\in Y$ such that
for every point $x\in \partial f^{-1}(y)$, there is an open
neighborhood $U_{x}$ of $x$ such that $P\not\subseteq f(U_{x})$ for
every $P\in \mathcal{P}_{y}$. Then $\partial
f^{-1}(y)\subset\cup\{U_{x}: x\in
\partial f^{-1}(y)\}$. Since $\partial f^{-1}(y)$ is compact, there exists a
finite subfamily $\mathcal{U}\subset \{U_{x}: x\in
\partial f^{-1}(y)\}$ such that $\partial f^{-1}(y)\subset\cup\mathcal{U}$. We denote
$\mathcal{U}$ by $\{U_{i}: 1\leq i\leq n_{0}\}$. Assume that
$\mathcal{P}_{y}=\{P_{n}:n\in\mathbb{N}\}$ and
$\mathcal{W}_{y}=\{F_{n}=\bigcap_{i=1}^{n}P_{i}:n\in\mathbb{N}\}$.
It is obvious that $\mathcal{W}_{y}\subset \mathcal{P}_{y}$ and
$F_{n+1}\subset F_{n}$, for every $n\in \mathbb{N}$. For each $1\leq
m\leq n_{0}, n\in\mathbb{N}$, it follows that there exists $x_{n,
m}\in F_{n}\setminus f(U_{m})$. Then denote $y_{k}=x_{n, m}$, where
$k=(n-1)n_{0}+m$. Since $\mathcal{P}_{y}$ is a network at point $y$
and $F_{n+1}\subset F_{n}$ for every $n\in \mathbb{N}$, $\{y_{k}\}$
is a sequence converging to $y$ in $Y$. Because $f$ is a
sequence-covering map, $\{y_{k}\}$ is an image of some sequence
$\{x_{k}\}$ converging to $x\in \partial f^{-1}(y)$ in $X$. From
$x\in \partial f^{-1}(y)\subset\cup\mathcal{U}$ it follows that
there exists $1\leq m_{0}\leq n_{0}$ such that $x\in U_{m_{0}}$.
Therefore, $\{x\}\cup\{x_{k}: k\geq k_{0}\}\subset U_{m_{0}}$ for
some $k_{0}\in \mathbb{N}$. Hence $\{y\}\cup\{y_{k}: k\geq
k_{0}\}\subset f(U_{m_{0}})$. However, we can choose an $n> k_{0}$
such that $k=(n-1)n_{0}+m_{0}\geq k_{0}$ and $y_{k}=x_{n, m_{0}}$,
which implies that $x_{n, m_{0}}\in f(U_{m_{0}})$. This
contradictions to $x_{n, m_{0}}\in F_{n}\setminus f(U_{m_{0}})$.
\end{proof}

The next lemma is obvious.

\begin{lemma}\label{l11}
Let $f:X\rightarrow Y$ be 1-sequence-covering, where $X$ is
$snf$-countable. Then $Y$ is $snf$-countable.
\end{lemma}

\begin{theorem}\label{t7}
Let $f:X\rightarrow Y$ be a sequence-covering and boundary-compact
map, where $X$ is first-countable. Then $Y$ is $snf$-countable if
and only if $f$ is an 1-sequence-covering map.
\end{theorem}

\begin{proof}
Necessity. Let $y$ be a non-isolated point in $Y$. Since $Y$ is
$snf$-countable, it follows from Lemma~\ref{l5} that there exists a
point $x_{y}\in\partial f^{-1}(y)$ such that whenever $U$ is an open
neighborhood of $x_{y}$, there is a $P\in \mathcal{P}_{y}$
satisfying $P\subset f(U)$. Let $\{B_{n}: n\in \mathbb{N}\}$ be a
countably neighborhood base at point $x_{y}$ such that
$B_{n+1}\subset B_{n}$ for each $n\in\mathbb{N}$. Suppose that
$\{y_{n}\}$ is a sequence in $Y$, which converges to $y$. Next, we
take a sequence $\{x_{n}\}$ in $X$ as follows.

Since $B_{n}$ is an open neighborhood of $x_{y}$, it follows from
the Lemma~\ref{l5} that there exists a $P_{n}\in\mathcal{P}_{y}$
such that $P_{n}\subset f(B_{n})$ for each $n\in \mathbb{N}$.
Because every $P\in\mathcal{P}_{y}$ is a sequential neighborhood, it
is easy to see that for each $n\in\mathbb{N}$, $f(B_{n})$ is a
sequential neighborhood of $y$ in $Y$. Therefore, for each
$n\in\mathbb{N}$, there is an $i_{n}\in\mathbb{N}$ such that
$y_{i}\in f(B_{n})$ for every $i\geq i_{n}$. Suppose that
$1<i_{n}<i_{n+1}$ for every $n\in \mathbb{N}$. Hence, for each $j\in
\mathbb{N}$, we take
\[x_{j}\in\left\{
\begin{array}{lll}
f^{-1}(y_{j}), & \mbox{if } j<i_{1},\\
f^{-1}(y_{j})\cap B_{n}, & \mbox{if } i_{n}\leq
j<i_{n+1}.\end{array}\right.\] We denote $S=\{x_{j}:j\in
\mathbb{N}\}$. It is easy to see that $S$ converges to $x_{y}$ in
$X$ and $f(S)=\{y_{n}\}$. Therefore, $f$ is an 1-sequence-covering
map.

Sufficiency. It easy to see that $Y$ is $snf$-countable by
Lemma~\ref{l11}.
\end{proof}

We don't know whether, in Theorem~\ref{t7}, $f$ is an
1-sequence-covering map when $X$ is only first-countable. However,
we have the following Theorem~\ref{t0}, which gives an affirmative
answer for Question 1.1. Firstly, we give some technique lemmas.

\begin{lemma}\cite{MN}\label{l0}
If $X\in {\it \Omega}$, then every compact subset of $X$ has a
countably external base.
\end{lemma}

\begin{lemma}\label{l1}
Let $f:X\rightarrow Y$ be a sequence-covering and boundary-compact
map. If $X\in {\it \Omega}$, then $Y$ is $snf$-countable.
\end{lemma}

\begin{proof}
Let $y$ be a non-isolated point for $Y$. Then $\partial f^{-1}(y)$
is non-empty and compact for $X$. Therefore, $\partial f^{-1}(y)$
has a countably external base $\mathcal{U}$ in $X$ by
Lemma~\ref{l0}. Let
$$\mathcal{V}=\{\cup\mathcal{F}:\mbox{There is a finite subfamily}\
\mathcal{F}\subset \mathcal{U}\ \mbox{with}\ \partial
f^{-1}(y)\subset\cup\mathcal{F}\}.$$ Obviously, $\mathcal{V}$ is
countable. We now prove that $f(\mathcal{V})$ is a countable
$sn$-network at point $y$.

(1) $f(\mathcal{V})$ is a network at $y$.

Let $y\in U$. Obviously, $\partial f^{-1}(y)\subset f^{-1}(U)$. For
each $x\in\partial f^{-1}(y)$, there exist an $U_{x}\in\mathcal{U}$
such that $x\in U_{x}\subset f^{-1}(U)$. Therefore, $\partial
f^{-1}(y)\subset\cup\{U_{x}: x\in\partial f^{-1}(y)\}$. Since
$\partial f^{-1}(y)$ is compact, it follows that there exists a
finite subfamily $\mathcal{F}\subset\{U_{x}: x\in\partial
f^{-1}(y)\}$ such that $\partial
f^{-1}(y)\subset\cup\mathcal{F}\subset f^{-1}(U)$. It is easy to see
that $F\in\mathcal{V}$ and $y\in\cup f(\mathcal{F})\subset U$.

(2) For any $P_{1}, P_{2}\in f(\mathcal{V})$, there exists a
$P_{3}\in f(\mathcal{V})$ such that $P_{3}\subset P_{1}\cap P_{2}$.

It is obvious that there exist $V_{1}, V_{2}\in \mathcal{V}$ such
that $f(V_{1})=P_{1}, f(V_{2})=P_{2}$, respectively. Since $\partial
f^{-1}(y)\subset V_{1}\cap V_{2}$, it follows from the similar proof
of (1) that there exists a $V_{3}\in\mathcal{V}$ such that $\partial
f^{-1}(y)\subset V_{3}\subset V_{1}\cap V_{2}$. Let
$P_{3}=f(V_{3})$. Hence $P_{3}\subset f(V_{1}\cap V_{2})\subset
f(V_{1})\cap f(V_{2})=P_{1}\cap P_{2}$.

(3) For each $P\in f(\mathcal{V})$, $P$ is a sequential neighborhood
of $y$.

Let $\{y_{n}\}$ be any sequence in $Y$ which converges to $y$ in
$Y$. Since $f$ is a sequence-covering map, $\{y_{n}\}$ is the image
of some sequence $\{x_{n}\}$ converging to $x\in\partial
f^{-1}(y)\subset X$. It follows from $P\in f(\mathcal{V})$ that
there exists a $V\in \mathcal{V}$ such that $P=f(V)$. Therefore,
$\{x_{n}\}$ is eventually in $V$, and this is implied that
$\{y_{n}\}$ is eventually in $P$.

Therefore, $f(\mathcal{V})$ is a countable $sn$-network at point
$y$.
\end{proof}

\begin{theorem}\label{t0}
Let $f:X\rightarrow Y$ be a sequence-covering and boundary-compact
map. If $X\in {\it \Omega}$, then $f$ is an 1-sequence-covering map.
\end{theorem}

\begin{proof}
From Lemma~\ref{l1} it follows that $Y$ is $snf$-countable.
Therefore, $f$ is an 1-sequence-covering map by Theorem~\ref{t7}.
\end{proof}

By Theorem~\ref{t0}, it easily follows the following
Corollary~\ref{c0}, which gives an affirmative answer for Question
1.1.

\begin{corollary}\label{c0}
Let $f:X\rightarrow Y$ be a sequence-covering and boundary-compact
map. Suppose also that at least one of the following conditions
holds:
\begin{enumerate}
\item $X$ has a point-countable base;

\item $X$ is a developable space.
\end{enumerate}
Then $f$ is
an 1-sequence-covering map.
\end{corollary}

\begin{lemma}\label{l6}
Let $f:X\rightarrow Y$ be a sequence-covering map, where $Y$ is
$snf$-countable and $\partial f^{-1}(y)$ has a countably external
base for each point $y\in Y$. Then, for each non-isolated point
$y\in Y$, there exists a point $x_{y}\in\partial f^{-1}(y)$ such
that whenever $U$ is an open subset with $x_{y}\in U$, there exists
a $P\in\mathcal{P}_{y}$ satisfying $P\subset f(U)$
\end{lemma}

\begin{proof}
Suppose not, there exists a non-isolated point $y\in Y$ such that
for every point $x\in
\partial f^{-1}(y)$, there is an open neighborhood $U_{x}$ of $x$
such that $P\not\subseteq f(U_{x})$ for every $P\in
\mathcal{P}_{y}$. Let $\mathcal{B}$ be a countably external base for
$\partial f^{-1}(y)$. Therefore, for each $x\in \partial f^{-1}(y)$,
there exists a $B_{x}\in \mathcal{B}$ such that $x\in B_{x}\subset
U_{x}$. For each $x\in
\partial f^{-1}(y)$, it follows that $P\not\subseteq f(B_{x})$
whenever $P\in \mathcal{P}_{y}$. Assume that
$\mathcal{P}_{y}=\{P_{n}:n\in\mathbb{N}\}$ and
$\mathcal{W}_{y}=\{F_{n}=\bigcap_{i=1}^{n}P_{i}:n\in\mathbb{N}\}$.
We denote $\{B_{x}\in\mathcal{B}:x\in\partial f^{-1}(y)\}$ by
$\{B_{m}:m\in\mathbb{N}\}$. For each $n, m\in\mathbb{N}$, it follows
that there exists $x_{n, m}\in F_{n}\setminus f(B_{m})$. For $n\geq
m$, we denote $y_{k}=x_{n, m}$ with $k=m+n(n-1)/2$. Since
$\mathcal{P}_{y}$ is a network at point $y$ and $F_{n+1}\subset
F_{n}$ for every $n\in \mathbb{N}$, $\{y_{k}\}$ is a sequence
converging to $y$ in $Y$. Because $f$ is a sequence-covering map,
$\{y_{k}\}$ is an image of some sequence $\{x_{k}\}$ converging to
$x\in \partial f^{-1}(y)$ in $X$. From $x\in \partial
f^{-1}(y)\subset\cup\{B_{m}: m\in \mathbb{N}\}$ it follows that
there exists a $m_{0}\in \mathbb{N}$ such that $B_{m_{0}}$ is an
open neighborhood at $x$. Therefore, $\{x\}\cup\{x_{k}: k\geq
k_{0}\}\subset B_{m_{0}}$ for some $k_{0}\in \mathbb{N}$. Hence
$\{y\}\cup\{y_{k}: k\geq k_{0}\}\subset f(B_{m_{0}})$. However, we
can choose a $k\geq k_{0}$ and an $n\geq m_{0}$ such that
$y_{k}=x_{n, m_{0}}$, which implies that $x_{n, m_{0}}\in
f(B_{m_{0}})$. This contradictions to $x_{n, m_{0}}\in
F_{n}\setminus f(B_{m_{0}})$.
\end{proof}

\begin{theorem}\label{t1}
Let $f:X\rightarrow Y$ be a sequence-covering and boundary-separable
map. If $X$ has a point-countable base and $Y$ is $snf$-countable,
then $f$ is an 1-sequence-covering map.
\end{theorem}

\begin{proof}
Obviously, $\partial f^{-1}(y)$ has a countably external base for
each point $y\in Y$. Therefore, it is easy to see by Lemma~\ref{l6}
and the proof of Theorem~\ref{t7}.
\end{proof}

{\bf Remark} We can't omit the condition ``$Y$ is $snf$-countable''
in Theorem~\ref{t1}. Indeed, the sequence fan
$S_{\omega}$\footnote{$S_{\omega}$ is the space obtained from the
topological sum of $\omega$ many copies of the convergent sequence
by identifying all the limit points to a point.} is the image of
metric spaces under the sequence-covering $s$-maps by
\cite[Corollary 2.4.4]{Ls3}. However, $S_{\omega}$ is not
$snf$-countable, and therefore, $S_{\omega}$ is not the image of
metric spaces under an 1-sequence-covering map.

In this section, we finally give an affirmative answer for Question
1.2.

\begin{lemma}\cite{BS}\label{l2}
Let $f:X\rightarrow Y$ be a map. If $X$ is a Fr\'echet
space\footnote{$X$ is said to be a {\it Fr\'echet space}\cite{Fr} if
$x\in\overline{P}\subset X$, there is a sequence in $P$ converging
to $x$ in $X$.}, then $f$ is a pseudo-open map\footnote{$f$ is a
{\it pseudo-open map}\cite{Ar2} if whenever $f^{-1}(y)\subset U$
with $U$ open in $X$, then $y\in\mbox{Int}(f(U))$.} if and only if
$Y$ is a Fr\'echet space and $f$ is a sequentially quotient map.
\end{lemma}

\begin{theorem}
Let $f:X\rightarrow Y$ be a boundary-compact map. If $X\in {\it
\Omega}$, then $f$ is a sequentially quotient map if and only if it
is a pseudo-sequence-covering map.
\end{theorem}

\begin{proof}
First, suppose that $f$ is sequentially quotient. If $\{y_{n}\}$ is
a non-trivial sequence converging to $y_{0}$ in $Y$, put
$S_{1}=\{y_{0}\}\cup \{y_{n}:n\in \mathbb{N}\},\
X_{1}=f^{-1}(S_{1})$ and $g=f|_{X_{1}}$. Thus $g$ is a sequentially
quotient,\ boundary compact map.\ So $g$ is a pseudo-open map by
Lemma~\ref{l2}. Since $X\in {\it \Omega}$, let
$\{U_{n}\}_{n\in\mathbb{N}}$ be a decreasingly neighborhood base of
compact subset $\partial g^{-1}(y_{0})$ in $X_{1}$. Thus
$\{U_{n}\cup\mbox{Int}(g^{-1}(y_{0}))\}_{n\in\mathbb{N}}$ is a
decreasingly neighborhood base of $g^{-1}(y_{0})$ in $X_{1}$. Let
$V_{n}=U_{n}\cup\mbox{Int}(g^{-1}(y_{0}))$ for each $n\in
\mathbb{N}$. Then $y_{0}\in\mbox{Int}(g(V_{n}))$, thus there exists
an $i_{n}\in \mathbb{N}$ such that $y_{i}\in g(V_{n})$ for each
$i\geq i_{n}$, so $g^{-1}(y_{i})\cap V_{n}\neq \emptyset$. We can
suppose that $1<i_{n}<i_{n+1}$. For each $j\in \mathbb{N}$, we take
\[x_{j}\in\left\{
\begin{array}{lll}
f^{-1}(y_{j}), & \mbox{if } j<i_{1},\\
f^{-1}(y_{j})\cap V_{n}, & \mbox{if } i_{n}\leq
j<i_{n+1}.\end{array}\right.\] Let $K=\partial g^{-1}(y_{0})\cup
\{x_{j}:j\in\mathbb{N}\}$. Clearly, $K$ is a compact subset in
$X_{1}$ and $g(K)=S_{1}$. Thus $f(K)=S_{1}$. Therefore, $f$ is a
pseudo-sequence-covering map.

Conversely, suppose that $f$ is a pseudo-sequence-covering map. If
$\{y_n\}$ is a convergent sequence in $Y$, then there is a compact
subset $K$ in $X$ such that $f(K)=\overline{\{y_n\}}$. For each
$n\in\mathbb{N}$, take a point $x_n\in f^{-1}(y_n)\cap K$. Since $K$
is compact and metrizable, $\{x_n\}$ has a convergent subsequence
$\{x_{n_k}\}$. So $f$ is sequentially quotient.
\end{proof}

\begin{corollary}
Let $f:X\rightarrow Y$ be a boundary-compact map. Suppose also that
at least one of the following conditions holds:
\begin{enumerate}
\item $X$ has a point-countable base;

\item $X$ is a developable space.
\end{enumerate}
Then $f$ is a
sequentially quotient map if and only if it is a
pseudo-sequence-covering map.
\end{corollary}

\begin{question}\label{q1}
Let $f:X\rightarrow Y$ be a sequence-covering and boundary-compact
(or compact) map. Is $f$ an 1-sequence-covering map if one of the
following conditions is satisfied?
\begin{enumerate}
\item Every compact subset of $X$ is metrizable;

\item Every compact subset of $X$ has countable character.
\end{enumerate}

\end{question}

{\bf Remark} If $X$ satisfies the conditions (1) and (2) in
Question~\ref{q1}, then $f$ is an 1-sequence-covering map by
Theorem~\ref{t0}.

\vskip 0.5cm
\section{Sequence-covering maps on $g$-metrizable spaces}
In this section, we mainly discuss sequence-covering maps on spaces
with a specially weak base.

\begin{lemma}\label{l3}
Let $f:X\rightarrow Y$ be a sequence-covering and boundary-compact
map. For each non-isolated point $y\in Y$, there exist a point
$x\in\partial f^{-1}(y)$ and a decreasingly weak neighborhood base
$\{B_{xi}\}_{i}$ at $x$ such that for each $n\in \mathbb{N}$, there
are a $P\in\mathcal{P}_{y}$ and $i\in \mathbb{N}$ with $P\subset
f(B_{xi})$ if $X$ and $Y$ satisfy the following (1) and (2):
\begin{enumerate}
\item $Y$ is $snf$-countable;

\item Every compact subset of $X$ has a countably externally weak base in $X$.
\end{enumerate}
\end{lemma}

\begin{proof}
Suppose not, there exists a non-isolated point $y\in Y$ such that
for every point $x\in
\partial f^{-1}(y)$ and every decreasingly weak neighborhood
base $\{B_{xi}\}_{i}$ of $x$, there is an $n\in \mathbb{N}$ such
that $P\not\subseteq f(B_{xn})$ for every $P\in \mathcal{P}_{y}$.
Since $\partial f^{-1}(y)$ is compact, it follows that $\partial
f^{-1}(y)$ has a countably externally weak base $\mathcal{B}$ of
$X$. Without loss of generality, we can assume that $\mathcal{B}$ is
closed under finite intersections. Therefore, for each $x\in
\partial f^{-1}(y)$, there exists a $B_{x}\in \mathcal{B}$
such that $P\not\subseteq f(B_{x})$ for every $P\in
\mathcal{P}_{y}$. Next, using the argument from the proof of
Lemma~\ref{l6}, this leads to a contradiction.
\end{proof}

The following Lemma~\ref{l4} is easily to check, and hence we omit
it.

\begin{lemma}\label{l4}
Let $X$ have a compact-countable weak base. Then every compact
subset of $X$ has a countably externally weak base in $X$.
\end{lemma}

\begin{theorem}\label{t2}
Let $f:X\rightarrow Y$ be a sequence-covering and boundary-compact
map, where $X$ has a compact-countable weak base. Then $Y$ is
$snf$-countable if and only if $f$ is an 1-sequence-covering map.
\end{theorem}

\begin{proof}
Necessity. Let $y$ be a non-isolated point in $Y$. Since $X$ has a
compact-countable weak base, it follows from Lemmas~\ref{l3}
and~\ref{l4} that there exists a point $x_{y}\in\partial f^{-1}(y)$
and a decreasingly countably weak base $\{B_{n}: n\in \mathbb{N}\}$
at point $x_{y}$ such that for each $n\in \mathbb{N}$, there is a
$P\in \mathcal{P}_{y}$ satisfying $P\subset f(B_{n})$. Suppose that
$\{y_{n}\}$ is a sequence in $Y$, which converges to $y$. Then we
can take a sequence $\{x_{n}\}$ in $X$ by the similar argument from
the proof of Theorem~\ref{t7}. Therefore, $f$ is an
1-sequence-covering map.

Sufficiency. By Lemma~\ref{l11}, $Y$ is $snf$-countable.
\end{proof}

We don't know whether the condition ``compact-countable weak base''
on $X$ can be replaced by ``point-countable weak base'' in
Theorem~\ref{t2},

\begin{corollary}\label{c1}
Let $f:X\rightarrow Y$ be a sequence-covering and boundary-compact
map, where $X$ is $g$-metrizable. Then $Y$ is $snf$-countable if and
only if $f$ is an 1-sequence-covering map.
\end{corollary}

Each closed sequence-covering map on metric spaces is
1-sequence-covering \cite{YP}. Now, we improve the result in the
following theorem.

\begin{theorem}
Let $f:X\rightarrow Y$ be a closed sequence-covering map, where $X$
is $g$-metrizable. Then $f$ is an 1-sequence-covering map.
\end{theorem}

\begin{proof}
Since $X$ is $g$-metrizable and $f$ is a closed sequence-covering
map, $Y$ is $g$-metrizable\cite[Theroem 3.3]{LC}. Therefore, $f$ is
a boundary-compact map by \cite[Corollary 2.2]{LC}. Hence $f$ is an
1-sequence-covering map by Corollary~\ref{c1}.
\end{proof}

\begin{question}
Let $f:X\rightarrow Y$ be a sequence-covering and boundary-compact
 map. If $X$ is $g$-metrizable, then is $f$ an 1-sequence-covering map?
\end{question}

\vskip 0.5cm
\section{Closed sequence-covering maps}

Say that a Tychonoff space $X$ is {\it strongly monotonically
monolithic} \cite{TV} if, for any $A\subset X$ we can assign an
external base $\mathcal {O}(A)$ to the set $\overline{A}$ in such a
way that the following conditions are satisfied:

(a) $|\mathcal {O}(A)|\leq \mbox{max}\{|A|, \omega\}$;

(b) if $A\subset B\subset X$ then $\mathcal {O}(A)\subset\mathcal
{O}(B)$;

(c) if $\alpha$ is an ordinal and we have a family
$\{A_{\beta}:\beta <\alpha\}$ of subsets of $X$ such that $\beta
<\beta^{\prime}<\alpha$ implies $A_{\beta}\subset
A_{\beta^{\prime}}$ then $\mathcal {O}(\cup_{\beta
<\alpha}A_{\beta})=\cup_{\beta <\alpha}\mathcal {O}(A_{\beta})$.

From \cite[Proposition 2.5]{TV} it follows that a Tychonoff space
with a point-countable base is strongly monotonically monolithic.
Moreover, if $X$ is a strongly monotonically monolithic space, then
it is easy to see that $X\in {\it \Omega}$ by \cite[Theorem
2.7]{TV}.

\begin{lemma}\label{l9}
Let $f:X\rightarrow Y$ be a closed sequence-covering map, where $X$
is a strongly monotonically monolithic space. Then $Y$ contains no
closed copy of $S_{\omega}$.
\end{lemma}

\begin{proof}
 Suppose that $Y$ contains a closed copy of $S_{\omega}$, and that
$\{y\}\cup\{y_{i}(n):i, n\in\mathbb{N}\}$ is a closed copy of
$S_{\omega}$ in $Y$, here $y_{i}(n)\rightarrow y$ as
$i\rightarrow\infty$. For every $k\in\mathbb{N}$, put
$L_{k}=\cup\{y_{i}(n):i\in \mathbb{N}, n\leq k\}$. Hence $L_{k}$ is
a sequence converging to $y$. Let $M_{k}$ be a sequence of $X$
converging to $u_{k}\in f^{-1}(y)$ such that $f(M_{k})=L_{k}$. We
rewrite $M_{k}=\cup\{x_{i}(n, k):i\in \mathbb{N}, n\leq k\}$ with
each $f(x_{i}(n, k))=y_{i}(n)$.

Case 1: $\{u_{k}:k\in\mathbb{N}\}$ is finite.

There are a $k_{0}\in\mathbb N$ and an infinite subset
$\mathbb{N}_{1}\subset \mathbb{N}$ such that $M_{k}\rightarrow
u_{k_{0}}$ for every $k\in\mathbb{N}_{1}$, then $X$ contains a
closed copy of $S_{\omega}$. Hence $X$ is not first countable. This
is a contradiction.

Case 2: $\{u_{k}:k\in\mathbb{N}\}$ has a non-trivial convergent
sequence in $X$.

Without loss of generality, we suppose that $u_{k}\rightarrow u$ as
$k\rightarrow \infty$. Since $X$ is first-countable, let $\{U_{m}\}$
be a decreasingly and open neighborhood base of $X$ at point $u$
with $\overline{U}_{m+1}\subset U_{m}$. Then
$\bigcap_{m\in\mathbb{N}}U_{m}=\{u\}$. Fix $n$, pick $x_{i_{m}}(n,
k_{m})\in U_{m}\cap\{x_{i}(n, k_{m})\}_{i}$. We can suppose that
$i_{m}< i_{m+1}$. Then $\{f(x_{i_{m}}(n, k_{m}))\}_{m}$ is a
subsequence of $\{y_{i}(n)\}$. Since $f$ is closed, $\{x_{i_{m}}(n,
k_{m})\}_{m}$ is not discrete in $X$. Then there is a subsequence of
$\{x_{i_{m}}(n, k_{m})\}_{m}$ converging to a point $b\in X$ because
$X$ is a first-countable space. It is easy to see that $b=u$ by
$x_{i_{m}}(n, k_{m})\in U_{m}$ for every $m\in\mathbb{N}$. Hence
$x_{i_{m}}(n, k_{m})\rightarrow u$ as $m\rightarrow\infty$. Then
$\{u\}\cup\{x_{i_{m}}(n, k_{m}):n, m\in\mathbb{N}\}$ is a closed
copy of $S_{\omega}$ in $X$. Thus, $X$ is not first countable. This
is a contradiction.

Case 3: $\{u_{k}:k\in\mathbb{N}\}$ is discrete in $X$.

Let $B=\{u_{k}:k\in\mathbb{N}\}\cup\{M_{k}:k\in\mathbb{N}\}$. Since
$X$ is strongly monotonically monolithic, $\overline{B}$ is
metrizable. Hence there exists a discrete family
$\{V_{k}\}_{k\in\mathbb{N}}$ consisting of open subsets of
$\overline{B}$ with $u_{k}\in V_{k}$ for each $k\in\mathbb{N}$. Pick
$x_{i_{k}}(1, k)\in V_{k}\cap\{x_{i}(1, k)\}_{i}$ such that
$\{f(x_{i_{k}}(1, k))\}_{k}$ is a subsequence of $\{y_{i}(n)\}$.
Since $\{x_{i_{k}}(1, k)\}_{k}$ is discrete in $\overline{B}$,
$\{f(x_{i_{k}}(1, k))\}_{k}$ is discrete in $Y$. This is a
contradiction.

In a word, $Y$ contains no closed copy of $S_{\omega}$.
\end{proof}

\begin{lemma}\label{l10}
Let $f:X\rightarrow Y$ be a closed sequence-covering map, where $X$
is a strongly monotonically monolithic space. Then $\partial
f^{-1}(y)$ is compact for each point $y\in Y$.
\end{lemma}

\begin{proof}
From Lemma~\ref{l9} it follows that $Y$ contains no closed copy
$S_{\omega}$. Since $X$ is a strongly monotonically monolithic
space, every closed separable subset of $X$ is metirzable, and hence
is normal. Therefore, $\partial f^{-1}(y)$ is countable compact for
each point $y\in Y$ by \cite[Theorem 2.6]{LC}. From \cite[Theorem
2.7]{TV} it easily follows that every countable compact subset of
$X$ is compact.
\end{proof}

\begin{theorem}\label{t6}
Let $f:X\rightarrow Y$ be a closed sequence-covering map, where $X$
is a strongly monotonically monolithic space. Then $f$ is an
1-sequence-covering map.
\end{theorem}

\begin{proof}
It is easy to see by Lemma~\ref{l10} and Theorem~\ref{t0}.
\end{proof}

\begin{corollary}\label{t3}
Let $f:X\rightarrow Y$ be a closed sequence-covering map, where $X$
is a Tychonoff space with a point-countable base. Then $f$ is an
1-sequence-covering map.
\end{corollary}

In fact, we can replace ``Tychonoff'' by ``regular'' in
Corollary~\ref{t3}, and hence we have the following result.

\begin{corollary}\label{c3}
Let $f:X\rightarrow Y$ be a closed sequence-covering map, where $X$
is a regular space with a point-countable base. Then $f$ is an
1-sequence-covering map.
\end{corollary}

\begin{proof}
Since $X$ has a point-countable base and $f$ is a closed
sequence-covering map, $Y$ has a point-countable base by
\cite[Theorem 3.1]{LC}. Therefore, $f$ is a boundary-compact map by
\cite[Lemma 3.2]{LC1}. Hence $f$ is an 1-sequence-covering map by
Corollary~\ref{c0}.
\end{proof}

We don't know whether, in Corollary~\ref{c3}, the condition ``$X$
has a point-countable base'' can be replaced by ``$X\in {\it
\Omega}$''. So we have the following question.

\begin{question}
Let $f:X\rightarrow Y$ be a closed sequence-covering map. If $X\in
{\it \Omega}$ (and $X$ is regular), then is $f$ an
1-sequence-covering map?
\end{question}

\begin{corollary}\label{c2}
Let $f:X\rightarrow Y$ be a closed sequence-covering map, where $X$
is a strongly monotonically monolithic space. Then $f$ is an
almost-open map\footnote{$f$ is an {\it almost-open map}\cite{Ar1}
if there exists a point $x_{y}\in f^{-1}(y)$ for each $y\in Y$ such
that for each open neighborhood $U$ of $x_{y}$, $f(U)$ is a
neighborhood of $y$ in $Y$.}.
\end{corollary}

\begin{proof}
$f$ is an 1-sequence-covering map by Theorem~\ref{t6}. For each
point $y\in Y$, there exists a point $x_{y}\in f^{-1}(y)$ satisfying
the Definition 2.2(4). Let $U$ be an open neighborhood of $x_{y}$.
Then $f(U)$ is a sequential neighborhood of $y$. Indeed, for each
sequence $\{y_{n}\}\subset Y$ converging to $y$, there exists a
sequence $\{x_{n}\}\subset X$ such that $\{x_{n}\}$ converges to
$x_{y}$ and $x_{n}\in f^{-1}(y_{n})$ for each $n\in \mathbb{N}$.
Obviously, $\{x_{n}\}$ is eventually in $U$, and therefore,
$\{y_{n}\}$ is eventually in $f(U)$. Hence $f(U)$ is a sequential
neighborhood of $y$. Since $X$ is first-countable, $Y$ is a
Fr\'echet space. Then $f(U)$ is a neighborhood of $y$. Otherwise,
suppose $y\in Y\setminus\mbox{int}(f(U))$, and therefore, $y\in
\overline{Y\setminus f(U)}$. Since $Y$ is Fr\'echet, there exists a
sequence $\{y_{n}\}\subset Y\setminus f(U)$ converging to $y$. This
is a contradiction with $f(U)$ is a sequential neighborhood of $y$.
Therefore, $f$ is an almost-open map.
\end{proof}

{\bf Remark} In \cite{TV}, V. V. Tkachuk has proved that closed maps
don't preserve strongly monotonically monolithic spaces. However, if
perfect maps\footnote{A map $f$ is called {\it perfect} if $f$ is a
closed and compact map} preserve strongly monotonically monolithic
spaces, then it is easy to see that closed sequence-covering maps
preserve strongly monotonically monolithity by Lemma~\ref{l10}. So
we have the following two questions.

\begin{question}
Do closed sequence-covering maps (or an almost open and closed maps)
preserve strongly monotonically monolithity?
\end{question}

\begin{question}
Do perfect maps preserve strongly monotonically monolithity?
\end{question}

In \cite{TV}, V. V. Tkachuk has also proved that open and separable
maps preserve strongly monotonically monolithity. However, we have
the following result.

\begin{theorem}
Let $f:X\rightarrow Y$ be an open and closed map, where $X$ is a
strongly monotonically monolithic space. Then $Y$ is a strongly
monotonically monolithic space.
\end{theorem}

\begin{proof}
From \cite[Theorem 3.4]{LC} it follows that $f$ is a
sequence-covering map. Therefore, $\partial f^{-1}(y)$ is compact
for each point $y\in Y$ by Lemma~\ref{l10}. Then $\partial
f^{-1}(y)$ is metrizable by \cite[Theorem 2.7]{TV}, and hence it is
separable, for each point $y\in Y$. For each point $y\in Y$, if $y$
is a non-isolated point, let $A_{y}$ be a countable dense set in the
subspace $\partial f^{-1}(y)$; if $y$ is an isolated point, then we
choose a point $x_{y}\in f^{-1}(y)$ and let $A_{y}=\{x_{y}\}$.

Let $B\subset Y$. Put $A_{B}=\cup\{A_{y}: y\in B\}$ and
$\mathcal{N}(B)=\{f(W): W\in\mathcal{O}(A_{B})\}$. It is easy to see
that $\mathcal{N}(B)$ satisfies the conditions (a)-(c) of the
definition of strongly monotonically monolithity. Therefore, we only
need to prove that $\mathcal{N}(B)$ is an external base for
$\overline{B}$. For each point $y\in\overline{B}$, let $U$ be open
subset in $Y$ with $y\in U$.

Case 1: $y$ is a non-isolated point in $Y$.

Since $f$ is an open map, $\emptyset\neq f^{-1}(y)\subset
\overline{f^{-1}(B)}$, and hence $\partial f^{-1}(y)\subset
\overline{f^{-1}(B)}$. Take any point $x\in
\partial f^{-1}(y)$. Then $x\in \overline{A_{B}}$. Therefore, there
exists a $V\in \mathcal{O}(A_{B})$ such that $x\in V\subset
f^{-1}(U)$. So $W=f(V)\in\mathcal{N}(B)$ and $y\in W\subset U$.

Case 2: $y$ is an isolated point in $Y$.

It is easy to see that $\{y\}\in \mathcal{N}(B)$, and therefore,
$y\in \{y\}\subset U$.

In a word, $\mathcal{N}(B)$ is an external base for $\overline{B}$.
\end{proof}

Let $\mathcal{B}=\{B_{\alpha}:\alpha\in H\}$ be a family of subsets
of a space $X$. $\mathcal{B}$ is {\it point-discrete}  (or {\it
weakly hereditarily closure-preserving}) if $\{x_{\alpha}:\alpha\in
H\}$ is closed discrete in $X$, whenever $x_{\alpha}\in B_{\alpha}$
for each $\alpha\in H$.

It is well-known that metrizability, $g$-metrizability,
$\aleph$-spaces, and spaces with a point-countable base are
preserved by closed sequence-covering maps, see \cite{LC, YP}. Next,
we shall consider spaces with a $\sigma$-point-discrete $k$-network,
and shall prove that spaces with $\sigma$-point-discrete $k$-network
are preserved by closed sequence-covering maps. Firstly, we give
some technique lemmas.

\begin{lemma}\label{l7}
Let $X$ be an $\aleph_{1}$-compact space\footnote{A space $X$ is
called {\it $\aleph_{1}$-compact} if each subset of $X$ with a
cardinality of $\aleph_{1}$ has a cluster point.} with a
$\sigma$-point-discrete network. Then $X$ has a countable network.
\end{lemma}

\begin{proof}
Let $\mathcal{P}=\bigcup_{n\in\mathbb{N}}\mathcal{P}_{n}$ be a
$\sigma$-point-discrete network for $X$, where each
$\mathcal{P}_{n}$ is a point-discrete family for each
$n\in\mathbb{N}$. For each $n\in\mathbb{N}$, let
$$B_{n}=\{x\in X: |(\mathcal{P}_{n})_{x}|>\omega\}.$$

Claim 1: $\{P\setminus B_{n}: P\in\mathcal{P}_{n}\}$ is countable.

Suppose not, there exist an uncountable subset $\{P_{\alpha}:
\alpha<\omega_{1}\}\subset \mathcal{P}_{n}$ and $\{x_{\alpha}:
\alpha<\omega_{1}\}\subset X$ such that $x_{\alpha}\in
P_{\alpha}\setminus B_{n}$. Since $\mathcal{P}_{n}$ is a
point-discrete family and $X$ is $\aleph_{1}$-compact,
$\{x_{\alpha}: \alpha<\omega_{1}\}$ is countable. Without loss of
generality, we can assume that there exists $x\in X\setminus B_{n}$
such that each $x_{\alpha}=x$. Therefore, $x\in B_{n}$, a
contradiction.

Claim 2: For each $n\in\mathbb{N}$, $B_{n}$ is a countable and
closed discrete subspace for $X$.

For each $Z\subset B_{n}$ with $|Z|\leq\omega_{1}$. Let
$Z=\{x_{\alpha}: \alpha\in\bigwedge\}$. By the definition of $B_{n}$
and Well-ordering Theorem, it is easy to obtain by transfinite
induction that $\{P_{\alpha}: \alpha\in\bigwedge\}\subset
\mathcal{P}_{n}$ such that $x_{\alpha}\in P_{\alpha}$ and
$P_{\alpha}\neq P_{\beta}$ for each $\alpha\neq\beta$. Therefore,
$Z$ is a countable and closed discrete subspace for $X$. Hence
$B_{n}$ is a countable and closed discrete subspace.

For each $n\in\mathbb{N}$, let
$\mathcal{P}_{n}^{\prime}=\{P\setminus B_{n}:
P\in\mathcal{P}_{n}\}\cup\{\{x\}: x\in B_{n}\}$. Then
$\mathcal{P}_{n}^{\prime}$ is a countable family.

Obviously, $\bigcup_{n\in\mathbb{N}}\mathcal{P}_{n}^{\prime}$ is a
countable network for $X$.
\end{proof}

The proof of the following lemma is an easy exercise.

\begin{lemma}\label{l8}
Let $\{F_{\alpha}\}_{\alpha\in A}$ be a point-discrete family for
$X$ and countably compact subset $K\subset\bigcup_{\alpha\in
A}F_{\alpha}$. Then there exists a finite family
$\mathcal{F}\subset\{F_{\alpha}\}_{\alpha\in A}$ such that $K\subset
\cup\mathcal{F}$.
\end{lemma}

\begin{lemma}\label{t4}
Let $\mathcal{P}$ be a  family of subsets of a space $X$. Then
$\mathcal{P}$ is a $\sigma$-point-discrete
$wcs^{\ast}$-network\footnote{A family $\mathcal{P}$ of $X$ is
called a {\it $wcs^{\ast}$-network}\cite{LT} of $X$, if whenever a
sequence $\{x_{n}\}$ converges to $x\in U$ with $U$ open in $X$,
there are a $P\in\mathcal{P}$ and a subsequence $\{x_{n_{i}}\}$ of
$\{x_{n}\}$ such that $x_{n_{i}}\in P\subset U$ for each $n\in
\mathbb{N}$} for $X$ if and only if $\mathcal{P}$ is a
$\sigma$-point-discrete $k$-network\footnote{A family $\mathcal{P}$
of $X$ is called a {\it $k$-network}\cite{PO} if whenever $K$ is a
compact subset of $X$ and $K\subset U$ with $U$ open in $X$, there
is a finite subfamily $\mathcal{P}^{\prime}\subset \mathcal{P}$ such
that $K\subset \cup\mathcal{P}^{\prime}\subset U$.} for $X$.
\end{lemma}

\begin{proof}
Sufficiency. It is obvious. Hence we only need to prove the
necessity.

Necessity. Let $\mathcal{P}=\bigcup_{n\in\mathbb{N}}\mathcal{P}_{n}$
be a $\sigma$-point-discrete $wcs^{\ast}$-network, where each
$\mathcal{P}_{n}$ is a point-discrete family for each
$n\in\mathbb{N}$. Suppose that $K$ is compact and $K\subset U$ with
$U$ open in $X$. For each $n\in \mathbb{N}$, let
$$\mathcal{P}_{n}^{\prime}=\{P\in\mathcal{P}_{n}: P\subset U\}, F_{n}=\cup\mathcal{P}_{n}^{\prime}.$$
Then there exists $m\in \mathbb{N}$ such that
$K\subset\bigcup_{k\leq m}F_{k}$. Suppose not, there is a sequence
$\{x_{n}\}\subset K$ with $x_{n}\in K-\bigcup_{i\leq n}F_{i}$. By
Lemma~\ref{l7}, it is easy to see that $K$ is metrizable. Therefore,
$K$ is sequentially compact. It follows that there exists a
convergent subsequence of $\{x_{n}\}$. Without loss of generality,
we assume that $x_{n}\rightarrow x$. Since $\mathcal{P}$ is a
$wcs^{\ast}$-network, there exist a $P\in \mathcal{P}$, and a
subsequence $\{x_{n_{i}}\}$ of $\{x_{n}\}$ such that $\{x_{n_{i}}:
i\in \mathbb{N}\}\subset P\subset U$. Therefore, there exists $l\in
\mathbb{N}$ such that $P\in \mathcal{P}_{l}^{\prime}$. Choose $i>
l$, since $P\subset F_{l}$, $x_{n_{i}}\in F_{l}$, a contradiction.
Hence there exists $m\in \mathbb{N}$ such that
$K\subset\bigcup_{k\leq m}F_{k}$. By Lemma~\ref{l8}, there is a
finite family $\mathcal{P}^{\prime\prime}\subset\bigcup_{i\leq
m}\mathcal{P}_{i}^{\prime}$ such that $K\subset
\cup\mathcal{P}^{\prime\prime}\subset U$. Therefore, $\mathcal{P}$
is a $k$-network.
\end{proof}

\begin{theorem}
Closed sequence-covering maps preserve spaces with a
$\sigma$-point-discrete $k$-network.
\end{theorem}

\begin{proof}
It is easy to see that closed sequence-covering maps preserve spaces
with a $\sigma$-point-discrete $wcs^{\ast}$-network. Hence closed
sequence-covering maps preserve spaces with a
$\sigma$-point-discrete $k$-network by Lemma~\ref{t4}.
\end{proof}

\begin{question}
Do closed maps preserve spaces with a $\sigma$-point-discrete
$k$-network?
\end{question}

\vskip0.9cm

\end{document}